\documentclass{amsproc}
\usepackage{euscript, graphicx, epstopdf}
\usepackage{cases}
\usepackage{mathrsfs}
\usepackage{bbm}
\usepackage{amssymb}
\usepackage{txfonts}
\usepackage{amscd}
\usepackage{amsfonts,latexsym,amsmath,amsxtra,mathdots,amssymb,latexsym,mathabx}
\usepackage[all,cmtip]{xy}
\usepackage{color}
\usepackage{colordvi}
\usepackage{multicol}
\usepackage{hyperref}
\usepackage{tikz}
\usepackage{float}
\usepackage{setspace, floatflt}

\allowdisplaybreaks

\makeatletter
\newsavebox\myboxA
\newsavebox\myboxB
\newlength\mylenA

\newcommand*\xoverline[2][0.75]{%
	\sbox{\myboxA}{$\m@th#2$}%
	\setbox\myboxB\null% Phantom box
	\ht\myboxB=\ht\myboxA%
	\dp\myboxB=\dp\myboxA%
	\wd\myboxB=#1\wd\myboxA% Scale phantom
	\sbox\myboxB{$\m@th\overline{\copy\myboxB}$}%  Overlined phantom
	\setlength\mylenA{\the\wd\myboxA}%   calc width diff
	\addtolength\mylenA{-\the\wd\myboxB}%
	\ifdim\wd\myboxB<\wd\myboxA%
	\rlap{\hskip 0.5\mylenA\usebox\myboxB}{\usebox\myboxA}%
	\else
	\hskip -0.5\mylenA\rlap{\usebox\myboxA}{\hskip 0.5\mylenA\usebox\myboxB}%
	\fi}
\makeatother

\DeclareFontFamily{U}{matha}{\hyphenchar\font45}
\DeclareFontShape{U}{matha}{m}{n}{
	<5> <6> <7> <8> <9> <10> gen * matha
	<10.95> matha10 <12> <14.4> <17.28> <20.74> <24.88> matha12
}{}
\DeclareSymbolFont{matha}{U}{matha}{m}{n}

\DeclareMathSymbol{\Lt}{3}{matha}{"CE}
\DeclareMathSymbol{\Gt}{3}{matha}{"CF}

\DeclareSymbolFont{mathd}{OML}{ztmcm}{m}{it}
\DeclareMathSymbol{\valpha}{\mathord}{mathd}{11}

\newsavebox{\foobox}
\newcommand{\slantbox}[2][.3]
{%
	\mbox
	{%
		\sbox{\foobox}{#2}%
		\hskip\wd\foobox
		\pdfsave
		\pdfsetmatrix{1 0 #1 1}%
		\llap{\usebox{\foobox}}%
		\pdfrestore
	}%
}

\newcommand{\BC}{{\mathbb {C}}}

 \newcommand{\BR}{{\mathbb {R}}}

 \newcommand{\BZ}{{\mathbb {Z}}}

\newcommand{\GL}{{\mathrm {GL}}} 
\newcommand{\SL}{{\mathrm {SL}}} \newcommand{\PSL}{{\mathrm {PSL}}}

\newcommand{\SU}{{\mathrm{SU}}}

\newcommand{\ds}{\displaystyle}

\newcommand{\ra}{\rightarrow}

\def\-{^{-1}}
\def\lp {\left (}
\def\rp {\right )} 
\def\BCx{\BC^{\times}}
\def\BCO{\BC \smallsetminus \{0\}}

\renewcommand{\Re}{{\mathrm{Re}\,}}

\def\bfJ {\slantbox{$\mathbb{J}$\hskip 1 pt}}
\def\nwedge {\hskip - 2 pt \wedge \hskip - 2 pt }
\def\shskip{\hskip 0.5 pt}
\def\tw{\textit{w}}

\def\dx{d^{\times} \hskip -1 pt}

\newcommand{\pmtrix}[4]{\begin{pmatrix} #1 & #2  \\ #3 & #4 \end{pmatrix}}

\makeatletter
\g@addto@macro\normalsize{\setlength\abovedisplayskip{3pt}}
\makeatother

\makeatletter
\g@addto@macro\normalsize{\setlength\belowdisplayskip{3pt}}
\makeatother

\newcommand{\delete}[1]{}

\theoremstyle{plain}

\newtheorem{thm}{Theorem}[section] \newtheorem{cor}[thm]{Corollary}
\newtheorem{lem}[thm]{Lemma}

\newtheorem*{acknowledgement}{Acknowledgements}
\newtheorem {rem}[thm]{Remark}

\numberwithin{equation}{section}

\begin{document}

	\title{A Whittaker-Plancherel Inversion Formula for $\mathrm{SL}_2(\mathbb{C})$}

	\author{Zhi Qi}
	\address{School of Mathematical Sciences\\ Zhejiang University\\Hangzhou, 310027\\China}
	\email{zhi.qi@zju.edu.cn}

	\author{Chang Yang}
	\address{Key Laboratory of High Performance Computing and Stochastic Information Processing (HPCSIP)\\ College of Mathematics and Statistics \\ Hunan Normal University \\Changsha,  410081\\China}
	\email{cyang@hunnu.edu.cn}
		
	\thanks{The second author is supported by the Construct Program of the Key Discipline in Hunan Province.}


	\begin{abstract}
		In this paper, we establish a Whittaker-Plancherel inversion formula for $\mathrm{SL}_2(\mathbb{C})$ from the analytic perspective of  the Bessel transform of Bruggeman and Motohashi. The formula gives a decomposition of the Whittaker-Fourier coefficient of a compactly supported function  on $\mathrm{SL}_2(\mathbb{C})$ in terms of its Bessel coefficients attached to irreducible unitary tempered representations of $\mathrm{SL}_2(\mathbb{C})$.
	\end{abstract}
	
	\maketitle
	
	\section{Introduction}\label{sec: Intro}
	
	Let $G=\SL_2(\BC)$. Define
\begin{align*}
N=\left\{n(u)=\begin{pmatrix}
1 & u \\
  & 1
\end{pmatrix}  :\, u\in \BC\right\}, \hskip 10 pt A= \left\{ s(z)=\pmtrix{z}{}{}{z^{-1}}  :\, z\in \BCx \right\}.
\end{align*}
Let $B = A N$ be the Borel subgroup of $G$. Let %$\varw$ be the Weyl element, namely,
\begin{align*}
\varw = \begin{pmatrix}
 & -1 \\
1 & 
\end{pmatrix}.
\end{align*}

Let $ C_c^{\infty}(G)$ be the space of smooth and compactly supported functions on $G$.

For $\lambdaup\in \BCx$, define the additive character $\psi = \psi_{\lambdaup}$ by
\begin{align}\label{1eq: defn of psi}
\psi_{\lambdaup}(u)=e(\lambdaup u+\overline{\lambdaup u}),
\end{align}  
with the usual abbreviation $e(x)=e^{2\pi ix}$. %Clearly, $\psi$ may be regarded as a character of $N$.

For a function $f\in C_c^{\infty}(G)$ define the associated Whittaker function by \begin{align}\label{1eq: defn of Wf(g)}
W_f^{\psi} (g)  = \int_{\BC} f (n(u) g) \psi (u) d u, \hskip 10 pt g\in G.
\end{align}
where $d u$ is {\it twice}  the Lebesgue measure on $\BC$.
The {\it Whittaker-Fourier coefficient} of $f$ is defined by \footnote{In number theory, when $f$ is a function on $G$ (not compactly supported) that comes from a Maass cusp form, there is a similar definition of $W_{{\psi_{\lambdaup}}} (f) = W_f^{\psi_{\lambdaup}} (1)$   related to the $\lambdaup$-th Fourier coefficient of $f$.}
\begin{align}\label{1eq: Whittaker-Fourier coefficient}
W_{\psi}(f)= W^{\psi}_f (1) = \int_{\BC}f(n(u))\psi (u)d u.
\end{align}
%\footnote{Note that  $d u$ is not normalized to be self-dual with respect to $\psi$.}. 
	
Next, we introduce the {\it Bessel coefficients} of a function $f \in C_c^{\infty}(G)$. 

For  $t \in \BR$ and $m \in \BZ$, let $\pi_{it, \shskip m} $ denote the principal series representation of $G$ unitarily induced from the following character of $B$,
\begin{align*}
s (z) n(u) \ra |z|^{2 i t} (z/|z|)^{m}.
\end{align*}
Precisely, $\pi_{it, \shskip m} $ is the representation of $G$  by right shifts on the space of smooth functions $\varphi: G\ra \BC$ satisfying $$\varphi(s(z)n(u)g)= |z|^{2 i t + 2} (z/|z|)^{m} \varphi(g), \hskip 10 pt g\in G. $$ %with the action of $G$ the right translation. 
It is well known that these $\pi_{it, \shskip m} $ are all the irreducible unitary {\it tempered} representations of $G$, up to the equivalence $\pi_{it, \shskip m} \cong  \pi_{- it, \shskip - m} $. %The central character of $\pi =   \pi_{it, \shskip m}$ is given by $\omega_{\pi} (\pm 1) = (\pm 1)^{m} $. %Define
%\begin{align}\label{1eq: Plancherel measure}
%\mu (t, m) = 4 t^2 + m^2/4,
%\end{align}
%then  the Plancherel measure of $G$ is given by $d \mu (\pi_{it, \shskip m}) = \mu (t, m) d t$ on the parameter space $ \BR \times \BZ$.

Let $J_{\nu} (z)$ be the classical Bessel function of the first kind,
\begin{equation}\label{1eq: series expansion of J}
	J_{\nu} (z) = \sum_{n=0}^\infty \frac {(-)^n (z/2)^{\nu+2n } } {n! \Gamma (\nu + n + 1) }.
\end{equation} 
We first introduce % for $t \in \BR$ and  $m\in \BZ$, % we adopt the notation
\begin{align}\label{1eq: defn of J mu m (z)}
J_{it,\shskip m}(z)=J_{-  it - \frac{1}{2} m}(z)J_{-   it + \frac{1}{2} m} (\bar{z}).  
\end{align}  
The function $J_{it,\shskip m}(z)$ is well defined for $z \in \BCx$ in the sense that the expression on the right of \eqref{1eq: defn of J mu m (z)} is independent on the choice of the argument of $z$ modulo $2 \pi$. %Moreover,  $J_{it,\shskip m}(z)$ is an even or odd function if $m$ is even or odd, respectively.
Define
\begin{equation}\label{1eq: Bessel function}
\bfJ_{it,\shskip m}(z)=\left\{\begin{aligned}
& \dfrac{ 1}{i \sinh( \pi t)}\big(J_{i t, \shskip m}(  z)-J_{-i t,-m}( z) \big),\ & & \text{ if $m$ is even,} \\
& \dfrac{ 1 }{i \cosh( \pi t)} \big(J_{i t,\shskip m}(  z)+J_{-i t,-m}( z) \big),\ & &\text{ if $m$ is odd.}
\end{aligned}\right.  
\end{equation}  \footnote{The definition of $\bfJ_{it,\shskip m}(z)$ is slightly different  from that in \cite{Qi-II-G} or \cite{Chai-Qi-Bessel}. The new normalization here is only for notational convenience.}
It is understood that in the non-generic case when $t = 0$ and $m$ is even the right hand side should be replaced by its limit.

For $\pi=\pi_{it,\shskip m}$ and $\psi=\psi_{\lambdaup}$, we define the attached Bessel function $ j_{\pi,\shskip\psi}$, supported on the open Bruhat cell $B\varw B = N A \varw N$, such that
\begin{equation}\label{1eq: defn of j}
j_{\pi,\shskip \psi}(s(z)\varw)= 2 \pi^2 \left|\lambdaup z \right|^2 \bfJ_{it,\shskip m}(4 \pi \lambdaup z), 
\end{equation}         
and that $j_{\pi,\shskip \psi}$ is left and right $(\psi,N)$-equivalent, namely,
\begin{align}\label{1eq: Bessel function--NN rule}
j_{\pi,\shskip \psi}(n(u)s(z)\varw \shskip n(\varv))=\psi(u)\psi(\varv)j_{\pi,\shskip \psi}(s(z) \varw).
\end{align}
Let $f\in C_c^{\infty}(G)$. We now define the {\it Bessel coefficient} of $f$ attached to $\pi$ by
\begin{equation}\label{1eq: Bessel coefficient}
J_{\pi,\shskip \psi}(f)=\int_G f(g)j_{\pi,\shskip \psi}(g)dg,
\end{equation}
where $dg$ is the Haar measure on $G$ defined by $dg= |z|^{-4} \shskip d u \shskip d \varv \shskip \dx z  $ for the coordinates $g = n(u)s(z) \varw \shskip n(\varv)$ on the Bruhat cell, with $\dx z = |z|^{-2} d z $ as usual. The Bessel coefficient (Bessel distribution) $J_{\pi,\shskip \psi}(f)$
has been investigated in \cite{Chai-Qi-Bessel}. Some results of  \cite{Chai-Qi-Bessel} will be recollected in Appendix \ref{appendix: Bessel}, but at this moment one only needs to know that the integral in \eqref{1eq: Bessel coefficient} is convergent for all  $f\in C_c^{\infty}(G)$.

Our Whittaker-Plancherel inversion formula for $G= \SL_2 (\BC)$ is as follows.
\begin{thm}\label{thm: main}
Let notation be as above.	For $f\in C_c^{\infty}(G)$ we have
	\begin{align} \label{1eq: Main Theorem}
	W_{\psi_{\lambdaup}}(f) 
	= \frac 1 { 32 \pi^4 |\lambdaup|^2} \sum_{m =-\infty}^{\infty}\int_{-\infty}^{\infty} J_{\pi_{it,\shskip m},\shskip \psi_{\lambdaup}}(f)   ( t^2 + m^2/4   ) dt.
	\end{align}
\end{thm}

Observe that $J_{it,\shskip m}(z)$ is an even or odd function according as $m$ is even or odd. So the formula \eqref{1eq: Main Theorem} splits into a pair of similar formulae if parity conditions are imposed on $f$. Also note that $ \pi_{it,\shskip m} $ is trivial on the center if and only if $m$ is even, in which case $ \pi_{it,\shskip m} $ may be regarded as representation of $\PSL_2 (\BC)$. So the even case of \eqref{1eq: Main Theorem} is the following formula for $ G\shskip / \{ \pm 1 \} = \PSL_2 (\BC)$.

\begin{cor} \label{thm: even}
Suppose that $f\in C_c^{\infty}(G)$ is even, namely, $f(g)=f(-g)$ for all $g\in G$. Then
\begin{equation}
W_{\psi_{\lambdaup}} (f)= \frac{1}{ 32 \pi^4 |\lambdaup|^2} \sum_{d=-\infty}^{\infty} \int_{-\infty}^{\infty} J_{\pi_{it, \shskip 2 d}, \shskip \shskip \psi_{\lambdaup} }(f) \big(t^2+ d^2\big)dt.
\end{equation} 
\end{cor}

\subsection*{Remarks}

A similar Whittaker-Plancherel formula for $G = \SL_2 (\BR)$ was obtained by Baruch and Mao \cite{BaruchMao-Whittaker}. They employ an analytic Bessel inversion formula of Kuznetsov\footnote{According to an anonymous referee, this formula should indeed be attributed to Sears and Titchmarsh \cite[(4.4)-(4.7)]{Sears-Titchmarsh} (see also \cite[(4.14.1), (4.14.2)]{Titchmarsh-Eigenfunction}). The formula of Kuznetsov in \cite[Appendix]{Kuznetsov} however is slightly more general.} while developing his celebrated trace formula for $\PSL_2 (\BZ)$ in \cite{Kuznetsov}. Both our main results and methods are in parallel with those of  \cite{BaruchMao-Whittaker} although the analysis here is quite different. Likewise, the Bessel inversion formula that we use here was discovered by Bruggeman and Motohashi in their work on the Kuznetsov trace formula for $\PSL_2 (\BZ[i])$. 

Consider the $\widebar \psi$-Whittaker space $C_c^{\infty} (N\backslash G; \widebar \psi)$ of smooth functions $W : G \ra \BC$, compactly supported modulo $N$, satisfying 
\begin{align*}
W (n (u) g) = \widebar \psi (u) W (g), \hskip 10 pt g \in G.
\end{align*}
Then each Whittaker function in  $C_c^{\infty} (N\backslash G; \widebar \psi)$ is a $W^{\psi}_f$ associated to some $f \in C_c^{\infty} (G)$. Recall that $W_{\psi} (f) = W^{\psi}_f (1)$. On the other hand, in view of \eqref{1eq: Bessel function--NN rule}, we have
\begin{align*}
J_{\pi,\shskip \psi}(f) = \int_{\BC} \int_{\BCx}  W_f^{\psi} ( s(z) \varw \shskip n(\varv))j_{\pi,\shskip \psi}( s(z) \varw) \psi (\varv) |z|^{-4}  \dx z \shskip d \varv .
\end{align*}
Hence, by simple comparison, the reader may find that our formula is an analogue to the Plancherel formula of Harish-Chandra for $\SL_2 (\BC)$. See for example \cite[Theorem 11.2]{Knapp-Book}. In contrast, the formula for $\SL_2 (\BR)$ in \cite{BaruchMao-Whittaker} is an analogue of \cite[Theorem 11.6]{Knapp-Book}. %As explained in \cite[\S XI.2]{Knapp-Book}, the Plancherel formula  of Harish-Chandra for $\SL_2 (\BR)$ is harder to establish than that for $\SL_2 (\BC)$ as there are two non-conjugate Cartan subgroups in $\SL_2 (\BR)$. On the contrary, the analysis here for $\SL_2 (\BC)$ is more difficult than that for $\SL_2 (\BR)$ in \cite{BaruchMao-Whittaker}. See \S \ref{sec: comparison} for more details.

%Our point here is that the analysis is always   harder  while the representation theory is relatively simpler when working over $\BC$ instead of $\BR$. We have to deal with the additional complexity arising from the analysis on the circle $\{ z : |z| = 1 \} \subset \BCx$, like a radial integral or, in our case, an {\it infinite} sum  as in \eqref{1eq: Main Theorem}. The reader may find evidences for this general principle in \cite{Qi-Bessel}, \cite{BaruchMao-Real} and \cite{Chai-Qi-Bessel} (see also \cite{Qi-Sph,Qi-II-G}).

A general decomposition formula for the space  $L^2(N \backslash G; \psi)$ of square integrable $\psi $-Whittaker functions of a real reductive group $G$ was obtained in  \cite[\S 15.9]{RRG-II}. Our formula, in contrast, is  point-wise and
does not follow from his decomposition formula as indicated in \cite{BaruchMao-Whittaker}.  Moreover, our proof is completely different from the proof of Wallach, since we do not
use Harish-Chandra's Plancherel formula. In particular, our proof brings to the front objects and
tools which we think are interesting by themselves: Bessel functions and Bessel distributions of
representations, orbital integrals and the Bessel transform of Bruggeman and Motohashi.

Recently, it was pointed out in \cite{vdBK-Wallach} that a lemma in the last chapter of \cite{RRG-II} is not correct. In response, Wallach gave a fixed proof of his Whittaker-Plancherel theorem in \cite{Wallach-Correction}. He also presented a point-wise Whittaker-Plancherel formula for a $K$-finite  $f$ in Harish-Chandra's Schwartz space on $G$ (see \cite[Theorem 48]{Wallach-Correction}).  Moreover, the
Whittaker-Plancherel formula for a reductive $p$-adic group $G$ was developed by Sakellaridis, Venkatesh \cite[\S 6.3]{Sak-Venkatesh-Periods} and Delorme \cite{Delorme-W-P} by different methods. The proof of Sakellaridis and Venkatesh is relatively short and can be readily adapted to real groups as in the works of Beuzart-Plessis \cite{B-P-1,B-P-2} (see   \cite[Proposition 2.14.2]{B-P-2}).

Finally, some remarks on the technical details are in order. Our analysis is quite different from that in Baruch and Mao \cite{BaruchMao-Whittaker}. However, if we use the differential equation instead of the recurrence relations for Bessel functions, along with an observation of the referee, we would have a simpler proof in the $\mathrm{SL}_2 (\mathbb{R})$ case. Moreover, their estimates for the Bessel transforms may be considerably improved.  See \S \ref{sec: bounds for Gf}.

\subsection*{Assumption} Subsequently, we shall assume, as we may without essential loss of generality, that $\lambdaup = 1$ so that $\psi (u) = \psi_1 (u) = e (u + \overline u)$. It is only up to the conjugation by $ s (\sqrt{\lambdaup})$. Also, the $\psi$ will be suppressed from the notation for simplicity.	

In fact, for general $\lambdaup$, we fix a $\sqrt \lambdaup $ and set $f_{\lambdaup}(g)=f\bigl(s(\sqrt\lambdaup)^{-1}g \shskip s(\sqrt\lambdaup) \bigr)$. Then we readily see that $W_{\psi_{\lambdaup}}(f)=W_{\psi_1}(f_{\lambdaup}) / |\lambdaup|^{ 2}$. Also, %we have $O_{f,\shskip \psi_{\lambdaup}}(z)=|\lambdaup|^{-4}O_{f_{\lambdaup},\shskip \psi_1}(\lambdaup z)$. I
we may  verify that $J_{\pi,\shskip\psi_{\lambdaup}}(f)=J_{\pi,\shskip\psi_1}(f_{\lambdaup})$. Hence the formula \eqref{1eq: Main Theorem} in the general case follows from the case $\lambdaup=1$.

\begin{acknowledgement}
Thanks are due to the anonymous referee for many remarks and suggestions, and for a crucial observation {\rm(}Lemma {\rm\ref{lem: referee's observation}}{\rm)} that greatly simplifies our arguments in an earlier draft of the paper and that leads us to a considerably improved estimate for the Bessel transform. %	We thank the referee for careful 	readings, helpful comments and, in particular, an update of the current development on the Whittaker-Plancherel formula. 
\end{acknowledgement}

\section{Orbital integrals and the Bessel transform}

In this section, we start our study with the Bessel coefficients $J_{\pi }(f)$ on the right-hand side of \eqref{1eq: Main Theorem}. It will be shown how one can express $J_{\pi }(f)$ as the Bessel transform of an orbital integral for $f$ in a simple way. 

    \subsection{Orbital integrals} Recall that $\psi (u) = e (u + \overline u)$. 
    For $f\in C_c^{\infty}(G)$, define
    \begin{align}\label{2eq: definition--orbital integral}
    O_{f}(z)&=\iint f (n(u)s(z) \varw \shskip n(\varv) )\psi(u)\psi(\varv)du \shskip d\varv , \hskip 10 pt z \in \BCx.
    \end{align}
Since the integrand is compactly supported,  the integral in \eqref{2eq: definition--orbital integral} is absolutely convergent. Moreover, it is easy to check that $O_f (z)$ is a smooth function on $\BCx$ which vanishes around $z = 0$. These orbital integrals were investigated by Jacquet \cite[\S 7]{Jacquet-RTF} in a different context. 

By the definitions in  \eqref{1eq: Bessel coefficient} and \eqref{2eq: definition--orbital integral}, along with \eqref{1eq: Bessel function--NN rule}, we have
\begin{align*}
J_{\pi }(f)&=\int_G f(g)j_{\pi }(g)dg \\
&=  \sideset{}{}{\int \hskip -4pt \int \hskip -4pt \int}f(n(u)s(z)\varw \shskip n(\varv))j_{\pi} (n(u)s(z)\varw n(\varv))\cdot |z|^{-4}du \shskip d \varv \shskip \dx z  \\
&=\int  j_{\pi }(s(z)w)\left(\int \hskip -4pt \int f
 ( n(u)s(z)\varw \shskip n(\varv) ) \psi(u)\psi(\varv) du \shskip d \varv \right)|z|^{-4}\dx z  \\
&=\int j_{\pi}(s(z)\varw) O_{f}(z)|z|^{- 4}\dx z. 
\end{align*}
More explicitly, in view of \eqref{1eq: defn of j}, we have
\begin{align*}%\label{2eq:Bessel distribution=integral transform}
J_{\pi_{it,\shskip m} }(f) = 2 \pi^2 \int_{\BC \smallsetminus\{0\}} \bfJ_{it,\shskip m}(4 \pi z) O_{f}(z)|z|^{- 2}\dx z. 
\end{align*}
Now we set 
\begin{align}\label{2eq:definition--Gf}
	G_{f}(z)= O_{f}(z/4\pi) / |z|^{2} 
\end{align}
so that 
\begin{align}\label{2eq:Bessel distribution=BM of Orbit}
	J_{\pi_{it,\shskip m} }(f)=  32\pi^4 \int_{\BC \smallsetminus\{0\}} \bfJ_{it,\shskip m}( z) G_f(z)\dx z.  
\end{align}
This immediately brings into mind the Bessel transform of Bruggeman and Motohashi.

\subsection{The Bessel transform of Bruggeman and Motohashi}

%Let $G (z) $ be a smooth function on $\BC \smallsetminus \{0\}$ such that $G (z) $ vanishes in a neighborhood of $0$ and that
%$G (z) = O (1/ |z|^{A} )$ for some $ A > 0$. 
Let $G (z)$ be a function in the space $ L^1 (\BCO, \dx z )  $.
Following \cite{B-Mo, B-Mo2}, we define its Bessel transform by
\begin{align} \label{2eq:Bessel transform--definition}
\breve{G} (it ,  m)=\int_{\BC\smallsetminus\{0\}}  \bfJ_{it, \shskip m}(z) G(z) \dx z.
\end{align}
Since $ \bfJ_{it, \shskip m}(z)  $ is of order  $O (1)$ at $0$ and $O(1/|z|)$ at $\infty$ except for the non-generic case when $t =0$ and $m$ even (see \eqref{appendix: J < 1} and \eqref{appendix: J asymptotic}), it follows that these integrals are absolutely convergent if $(it, m)$ is generic.

%\red{add: convergence}
%Note that $\breve{G} (it ,  m)$ vanishes if $G (z)$ and $m$ have opposite parity.

The Bessel inversion formula of Bruggeman and Motohashi in \cite{B-Mo} was originally for $\PSL_2 (\BC)$. It was soon generalized to $\SL_2 (\BC)$ in the thesis of Lokvenec-Guleska \cite{B-Mo2}. We reformulate her formula in the following theorem.\footnote{Note that Lokvenec-Guleska normalizes the Bessel functions in an unnatural way to be all even.}

\begin{thm}[Bruggeman, Motohashi and Lokvenec-Guleska]\label{theorem::Bessel transform--Inversion}
	Assume that $G (z) \in  L^1 (\BCO, \dx z ) \cap L^2 (\BCO, \dx z)$ is continuous and that 
	%\begin{align}\label{1eq: condition 2 for G}
$$	\breve{G} (it , m) = O\bigl( 1/ (t^2+ m^2 + 1)^{q} (|m|+1) \bigr)$$  for some $q > 1$. 
Then 
	\begin{align}\label{2eq::Bessel transform--Inversion formula}
	G (z)=\frac{1}{8}\sum_{m = -\infty}^{\infty} \int_{  -\infty}^{\infty} (-)^m \bfJ_{it, \shskip m}(z)\breve{G} (it , m)  (t^2+m^2/4 )d t.
	\end{align} 
\end{thm}

\begin{cor}\label{cor::Bessel transform--Inversion}
	Assume that $G (z)$ is a  smooth function on $\BCO$   such that $G (z)  $ vanishes in a neighborhood of $0$ and  
	$G (z) = O (1/ |z|^{p} )$
	for some $ p > 0$. Assume also that 
	%\begin{align}\label{1eq: condition 2 for G}
$\breve{G} (it , m) = O\bigl( 1/ (t^2+ m^2 + 1)^{q} (|m|+1) \bigr)$  for some $q > 1$.
	%	\end{align} 
	Then the Bessel inversion formula {\rm\eqref{2eq::Bessel transform--Inversion formula}} holds for $G (z)$.
\end{cor}

\begin{rem}
The formula {\rm\eqref{2eq::Bessel transform--Inversion formula}} may be found in {\rm\cite[Theorem 11.1]{B-Mo}} and {\rm\cite[Theorem 12.2.1]{B-Mo2}}, but it is proven only for function $G (z)$ compactly supported on $\BCO$. It is not sufficient for our purpose. They however also proved a Parseval-Plancherel formula for the Bessel transform, and it would enable us to prove {\rm\eqref{2eq::Bessel transform--Inversion formula}} for $G(z)$  as in Theorem {\rm \ref{theorem::Bessel transform--Inversion}}.  %Nevertheless, it is easy to generalize it for the functions $G(z)$ as in Theorem {\rm\ref{theorem::Bessel transform--Inversion}} by approximation. 
This will be done in Appendix {\ref{appendix: Proof of B-M}}. 
\end{rem}

\begin{rem}
	This Bessel inversion formula is the bridge to the ``Kloosterman form" of the Kuznetsov trace formula for $\SL_2 (\BC)$ from its ``spectral form" in {\rm\cite{B-Mo,B-Mo2}}. There is however a direct  approach to the Kloosterman form by a representation theoretic method of Cogdell and Piatetski-Shapiro in {\rm\cite{Qi-Kuz}}. Compare also {\rm\cite{Kuznetsov}} and {\rm\cite{CPS}} for $\SL_2 (\BR)$.
\end{rem}

\begin{rem}
	 In {\rm\cite[Chapter 3]{Qi-Thesis}}, the Bessel inversion formula for $\SL_2 (\BC)$, along with the $\GL_3 (\BC) \times \GL_2 (\BC) $ local functional equations, is applied in a formal way to derive a formula of Bessel functions for $\GL_3 (\BC)$. A similar formula is obtained for $\GL_3 (\BR)$ by the Bessel inversion formula of Kontorovich, Lebedev and Kuznetsov {\rm(}Sears and Titchmarsh{\rm)}.
\end{rem}

One of the main technical results of this work is to show that the function $G_f (z)$ satisfies the conditions in Corollary \ref{cor::Bessel transform--Inversion} for every $f \in C_c^{\infty} (G)$. %; this will be done in \S \ref{sec: Jacquet} and \ref{sec: estimates for Gf}. 
For this, we shall first need the results of Jacquet \cite{Jacquet-RTF} on the asymptotic of  orbital integrals $O_f (z)$ as $|z| \ra \infty$. {In particular, Jacquet proves that the Whittaker coefficient $W (f)$ on the left-hand side of \eqref{1eq: Main Theorem} arises in the leading terms of the asymptotic formula for $O_f (z)$ (if $f$ is even or odd).}

\section{Asymptotic of orbital integrals}\label{sec: Jacquet}

We now recollect Jacquet's results on orbital integrals  \cite[\S 7]{Jacquet-RTF} as follows. Since it is so elegant, we are prompted to include %an outline of 
his proof here. 

\begin{thm}[Jacquet]\label{thm: Jacquet}
	Let $f\in C_c^{\infty}(G)$.  Then there are two compactly supported smooth functions $H_{  +}$ and $H_{ -}$ on $\BC$ such that 
	\begin{align}\label{3eq: Orbit integral--characterization}
		O_{f}(z)=\psi(2z) |z| H_{ +} (1/z)+ \psi(-2z) |z| H_{ -} (1/z),  
	\end{align}
	and
	\begin{align}\label{3eq: Orbit integral--first term}
		H_{+}(0) = W_f(1) / 2,\quad  H_{-}(0) = W_f(-1) / 2.
	\end{align}
\end{thm}

\begin{rem}
There is an alternative proof of the theorem by Baruch and Mao in {\rm\cite{BaruchMao-Whittaker}} for $G = \SL_2 (\BR)$, directly using the method of stationary phase.  It would however be quite technical to generalize their arguments to $\BC$. Jacquet's original proof applies the Parseval-Plancherel formula of Weil as a substitute for stationary phase. It works not only for $\BR$ and $\BC$ but also for non-Archimedean local fields. 
\end{rem}

\begin{proof}[Proof of Theorem \ref{thm: Jacquet}]
		Let 
		\begin{align*}
		N^- =\left\{n^{\scriptscriptstyle -}(u)=\begin{pmatrix}
		1 &  \\
		u & 1
		\end{pmatrix}  :\, u\in \BC\right\}.
		\end{align*}
		We have 
		$$G  = N A \varw N \cup  N^{-} A N. $$
		Thus we may assume that the support of $f$ is contained in one of these cells.
		Recall the definitions of  $W_f (g)$ and $O_f (z)$ in \eqref{1eq: defn of Wf(g)} and \eqref{2eq: definition--orbital integral}. If $f$ is supported on $NA\varw N$, then  $W_f (1) = W_f (-1) = 0$ and $O_f(z)$ is  compactly supported on $\BCx$. Now assume that the support of $f$ is contained in ${N}^- AN$. Then
		the function $T $ defined by	
		\begin{align*}
		T(u, z)=\int f \big(  n^{\scriptscriptstyle -} (u)  s (z) n (\varv ) \big) \psi(\varv) d \varv
		\end{align*}
		is  smooth and compactly supported on the product $\BC \times \BCx$. It is clear that
		\begin{align*}
		T (0, 1) = W_f (1), \hskip 10 pt 	T (0, - 1) = W_f (-1).
		\end{align*}
		
		Note that for $u \in \BCx$	we have 
		\begin{align*}
		\pmtrix{1}{u}{}{1} \pmtrix{z}{}{}{z^{-1}} \varw \pmtrix{1}{\varv}{}{1}       =  \pmtrix{1}{}{u^{-1}}{1}  \pmtrix{z^{-1}u}{}{}{z u^{-1}}  \pmtrix{1}{\varv-z^2 u^{-1}}{}{1}   .
		\end{align*}
		After suitable changes of variables,
		 $O_f (z)$ may be reformulated  as follows
		\begin{align*}
		O_f(z)= | z |^2 \int T \big((uz)^{-1} , u\big) \psi \big(z \big( u + u^{-1} \big) \big) d u.
		\end{align*}
Since the phase function $u + u^{-1}$ has two stationary points $1$ and $-1$, we may introduce a suitable partition of unity and write
		\begin{align*}
		O_f=K_0 + K_{+} + K_{-}
		\end{align*}
		so that $K_0  $ is a Schwartz function on $\BCx$ and that $K_{\pm} $ is of the form
		\begin{align*}
		K_{\pm}(z) = |z|^2 \int  T_{\pm} \big((uz)^{-1} , u\big) \psi \big(z \big( u + u^{-1} \big) \big) d u;
		\end{align*}
	the function $T_{\pm}$ is smooth of compact support on $\BC \times \BCx$; furthermore,	the projection of the support of $T_{\pm}$ to the second factor is contained in a small neighborhood of $\pm 1$ and 
	\begin{align*}
	T_{\pm} (0, \pm 1) = T (0, \pm 1) = W_f (\pm 1).
	\end{align*} In the integral for $K_{\pm}$, we first introduce a new variable given by
		\begin{align*}
		u + u^{-1} =   \tw^2 \pm 2,\quad  \tw = u^{1/2} \mp u^{-1/2},
		\end{align*}
		where the square root is chosen to be the principal branch.  We find then
		\begin{align*}
		K_{\pm}(z) = \psi(\pm 2z ) |z|^2  \int \Phi_{\pm} (z^{-1} ,\tw) \psi (z \tw^2 ) d \tw,
		\end{align*}
		with
		\begin{align*}
		\Phi_{\pm}(z , \tw) = T_{\pm} \big( z u^{-1} , u\big)  \left|  {d \tw}/{ d u}  \right| ^{-2}.
		\end{align*}
		At this point, we invoke the Parseval-Plancherel identity of Weil,
		\begin{align}\label{2eq: Weil}
		\int \widehat\Phi(u) \psi \big( \tfrac{1}{2} z u^2 \big)d u = |z|^{-1} \int \Phi(u) \psi
		\big(-\tfrac{1}{2} z^{-1} u^2 \big) d u,
		\end{align}
		in which $\Phi$ is a Schwartz function on $\BC$ and $\widehat\Phi$ is its Fourier transform 
		\begin{align*}
		\widehat\Phi (u) = \int \Phi (\varv) \psi (u \varv) d \varv.
		\end{align*}		
		As a consequence of (\ref{2eq: Weil}), we have 
		\begin{align*}
		K_{\pm}(z) =  \psi ( \pm 2z )  |z| H_{\pm}(1/z),
		\end{align*}
		where we have set
		\begin{align*}
		& H_{\pm} ( z ) =  \tfrac 1 2 \hskip -1 pt \int \widehat \Phi_{\pm}(  z ,   \tw ) \psi \big( - \tfrac 1 4 \tw^2 z \big) d \tw, \\
		&\widehat\Phi_{\pm} (z,\tw) = \int \Phi_{\pm} (z , \varv) \psi (\varv \tw) d \varv.
		\end{align*}
		Now $H_{\pm}$ is certainly smooth near $0$ and 
		\begin{align*}
	2	H_{\pm}( 0 ) =   \int \widehat \Phi_{\pm} ( 0,  \tw ) d \tw =  \Phi_{\pm} ( 0, 0)  =   T_{\pm} ( 0, \pm 1)   =  W_f(\pm 1)  .
		\end{align*}
\end{proof}

\subsection{Proof of Theorem \ref{thm: main}}\label{sec: proof of main theorem}

%Granted that Theorem \ref{thm: analytic} holds, we are now ready to prove Theorem \ref{thm: main}. Still one assumption is needed in the course of the proof. We leave its justification to \S \ref{sec: change of order}. 

We are now ready to prove Theorem \ref{thm: main}. We shall leave two questions to \S \ref{sec: estimates for Gf}. The first point is that our function $G_f (z)$ that was defined in \eqref{2eq:definition--Gf} satisfies the
conditions of Corollary \ref{cor::Bessel transform--Inversion} so that we can apply the Bessel inversion formula for $G_f (z)$.
The second is a change of order of a limit, summation and integration that will be pointed out in the course of the proof.

Let $f\in  C_c^{\infty}(G)$. We can write $f$ as the sum of an even and an odd function. So, by linearity, we may well assume that $f$ is itself even or odd. As the arguments for the two cases are similar, we shall only prove the even case, that is, Corollary \ref{thm: even}.

Now let $f$ be even. Then $O_f(z)$ and hence $G_f(z)$ are even functions on $\BCO$. Observe that
\begin{align*}
W_f (1) =W_f (-1) = W (f).
\end{align*}
It follows from \eqref{2eq:definition--Gf}, \eqref{3eq: Orbit integral--characterization} and \eqref{3eq: Orbit integral--first term} that
\begin{align*}%\label{3eq: Wf=lim Gf}
W(f)=\lim_{z\ra \infty}\frac{ 4 \pi |z| G_f(z)}{ \cos (z+\bar z) },
\end{align*}
where the limit is taken on a set of $z$ such that the denominator $\cos (z+\bar z)$ is bounded away from $0$. %(later in \S \ref{sec: estimates for Gf} we shall actually restrict the limit on the positive real line). 
%We now make our {\it first} assumption here that
By \eqref{2eq::Bessel transform--Inversion formula}, we have
\begin{align*}
G_f(z) = \frac{1}{8} \sum_{d=-\infty}^{\infty}\int_{-\infty}^{\infty} \bfJ_{it, \shskip 2d} (z)\breve{G}_f (it, 2d) \big( t^2 + d^2\big) dt. 
\end{align*}
Hence %$W(f)$ is equal to
\begin{align}\label{3eq: lim sum and int}
W(f) = \lim_{z\ra \infty} \frac{   \pi |z|}{ 2\cos (z+\bar{z}) }  \sum_{d=-\infty}^{\infty}\int_{-\infty}^{\infty} \bfJ_{it, \shskip 2d} (z)\breve{G}_f (it, 2d) \big( t^2 + d^2 \big) dt. 
\end{align}
Assume now that it is legitimate to interchange the order of the limit, summation and integration in \eqref{3eq: lim sum and int}. This will be justified using the dominated
convergence theorem in \S   \ref{sec: change of order}. We then have
\begin{align*}
W(f)=\sum_{d=-\infty}^{\infty}\int_{-\infty}^{\infty}\lim_{z\ra \infty} \frac{   \pi |z|\bfJ_{it, \shskip 2d} (z)}{2 \cos (z+\bar{z}) } \breve{G}_f (it, 2d) \big( t^2 + d^2 \big) dt . 
\end{align*}
It follows from the asymptotic formula for $\bfJ_{it, \shskip 2 d}(z)$ (see \eqref{appendix: J asymptotic}) that
\begin{align*}
\lim_{z\ra \infty}\frac{   \pi |z| \bfJ_{it, \shskip 2d}(z)}{2 \cos (z+\bar{z}) }= 1 . 
\end{align*}
Consequently,
\begin{align*}
W(f)=   \sum_{d=-\infty}^{\infty}\int_{-\infty}^{\infty}  \breve{G}_f (it, 2d) \big( t^2 + d^2 \big) dt .
\end{align*}	
In view of \eqref{2eq:Bessel distribution=BM of Orbit} and \eqref{2eq:Bessel transform--definition}, we conclude that
\begin{align*}
W(f)= \frac{1}{32 \pi^4}\sum_{d=-\infty}^{\infty}\int_{-\infty}^{\infty} J_{\pi_{it,\shskip 2d} }(f) \big( t^2 + d^2 \big) dt .
\end{align*}

\section{\texorpdfstring{Estimates for the Bessel function $\bfJ_{it, \shskip m} (z)$}{Estimates for the Bessel function $J_{it, \shskip m} (z)$}}\label{sec: estinates for J}

\subsection{\texorpdfstring{Preliminaries on the Bessel function $\bfJ_{\mu, \shskip m} (z)$}{Preliminaries on the Bessel function $J_{\mu, \shskip m} (z)$}}

In this section, we recollect some results on the Bessel function $\bfJ_{\mu, \shskip m} (z)$ for $\mu \in \BC$ and $m \in \BZ$. We do not feel it is necessary to restrict ourselves to Bessel functions  $\bfJ_{i t, \shskip m} (z)$ of pure imaginary order.

Let $J_{\nu} (z)$ be the Bessel function of the first kind of order $\nu$. Let $H_{\nu}^{(1)} (z)$ and $H_{\nu}^{(2)} (z)$ be the Hankel functions of order $\nu$. They all satisfy the Bessel differential equation
\begin{align}\label{appendix: Bessel DE}
z^2 \frac {d^2 \tw} {d z^2} + z \frac {d \tw} {d z} + \big(z^2 - \nu^2 \big) w = 0.
\end{align}

\subsubsection{}
We have the definition
\begin{equation}\label{4eq: defn of JJ mu m}
	\bfJ_{\mu ,\shskip m}(z)= \left\{\begin{aligned}
& \dfrac{ 1}{ \sin( \pi \mu )}\big(J_{\mu, \shskip m}(  z)-J_{- \mu,-m}( z) \big),\ && \text{ if $m$ is even,} \\
& \dfrac{ 1}{i\cos( \pi \mu )} \big(J_{\mu,\shskip m}(  z)+J_{- \mu,-m}( z) \big),\ & &\text{ if $m$ is odd.}
	\end{aligned}\right. 
\end{equation} 
with 
\begin{equation}\label{4eq: defn of J mu m}
J_{\mu,\shskip m}(z)=J_{-  \mu - \frac{1}{2} m}(z)J_{-  \mu  + \frac{1}{2} m} (\bar{z}).
\end{equation}

%When  $\nu \neq - 1, -2, -3, ...$, w
It is readily seen from the series expansion of $J_{\nu} (z)$ (see \eqref{1eq: series expansion of J}) that 
\begin{align*}
	J_{\nu} (z)  \Lt_{\nu}  |z^{\nu}|, \hskip 10pt |z| \leqslant 1. 
\end{align*}
It follows that if $\mu$ is generic, that is $2 \mu \notin 2 \BZ + m$, then
\begin{align}\label{2eq: bound for J mu m}
\bfJ_{ \mu, \shskip  m} (z)   \Lt_{\shskip \mu,\shskip  m} \left| |z|^{- 2 \mu} \right| + \left| |z|^{ 2 \mu} \right|, \hskip 10 pt |z| \leqslant 1.
\end{align}
In particular, 
\begin{align}\label{appendix: J < 1}
	\bfJ_{it, \shskip m} (z) \Lt_{\shskip t,\shskip m} 1,\hskip 10pt  |z|\leqslant 1   ,
\end{align}
except for the non-generic case when $t=0$ and $m$ even. In the non-generic case, we have a slightly worse estimate (see \cite[(2.24)]{Qi-II-G})
\begin{align}
	\bfJ_{0, \shskip m} (z) \Lt_{\shskip  m} \log (2/|z|) ,\hskip 10pt  |z|\leqslant 1, \hskip 10 pt \text{($m$ even)}.
\end{align}

We have a second expression of $\bfJ_{\mu, \shskip m}(z)$ in terms of Hankel functions. Define
\begin{align*}
	H_{\mu,\shskip m}^{(1,\shskip 2)}(z) = H^{(1,\shskip 2)}_{ \mu + \frac{1}{2} m}(z) H^{(1,\shskip 2)}_{ \mu-\frac{1}{2}m}\left(\bar{z}\right).
\end{align*} 
By the connection formulae \cite[3.6 (1),(2)]{Watson}, we have
\begin{align*}
	\bfJ_{\mu, \shskip m}(z)=\frac{i}{\hskip 1 pt 2 \hskip 0.5 pt}\left( (-)^m e^{ \pi i\shskip  \mu }H^{(1)}_{\mu,\shskip m}(z) - e^{ - \pi i \shskip \mu}H^{(2)}_{\mu,\shskip m}(z)\right).
\end{align*}
From the asymptotic formula of Hankel functions \cite[7.2 (5),(6)]{Watson}, we deduce that
\begin{align}\label{appendix: J asymptotic}
	\bfJ_{\mu, \shskip m}(z) \sim {\frac{1}{\pi  |z|}} \lp (-)^m  e \left(   ( z + \bar{z} ) / 2 \pi \right) + e \left( - ( z + \bar{z} ) / 2 \pi \right) \rp, \hskip 10 pt |z| \ra \infty.
\end{align}
It is then clear that
\begin{align}\label{4eq: J(z)=O(1/|z|)}
	  \bfJ_{\mu, \shskip m} (z)   = O  (1 / |z|) , \hskip 15 pt |z| \ra \infty.
\end{align}
Moreover, it follows from \cite[\S 7.13.1]{Olver} that
\begin{align}\label{4eq: J(z) <1/z}
\bfJ_{\mu, \shskip m}(z) \Lt 1 /|z|, \hskip 15 pt |z| > |\mu|^2 + m^2 + 1,
\end{align}
where the implied constant is absolute.

\subsubsection{Recurrence formulae}
Recall the following recurrence formula \cite[3.2 (2)]{Watson},
\begin{align}\label{4eq: recurrence Jv} 
2 J'_{\nu}(z)  =  J_{\nu - 1} (z) - J_{\nu + 1} (z) . 
\end{align}
By \eqref{4eq: recurrence Jv}, it is straightforward to derive the corresponding recurrence formulae
\begin{equation}\label{4eq: recurrence partial} 
		\begin{split}
	&	2 i \partial \bfJ_{\mu, \shskip m} (z) / \partial z  = \bfJ_{\mu + \frac{1}{2}, \, m + 1} (z) + \bfJ_{\mu - \frac{1}{2}, \, m - 1} (z) , \\
		&	2 i \partial \bfJ_{\mu, \shskip m} (z) / \partial \bar{z}  = \bfJ_{\mu + \frac{1}{2}, \, m - 1} (z) + \bfJ_{\mu - \frac{1}{2}, \, m + 1} (z) . 
		\end{split}
\end{equation}
%\begin{align}\label{4eq: recurrence division by z}
%\begin{cases}
%\frac{ - 4 \mu - m}{ z } \bfJ_{\mu, \shskip m} (z) =  -i \shskip \bfJ_{\mu + \frac{1}{4}, \shskip m + 1} (z)   +   i \shskip \bfJ_{\mu - \frac{1}{4}, \shskip m - 1} (z)  \\[10pt]
%\frac{ - 4 \mu + m}{ \bar{z} } \bfJ_{\mu, \shskip m} (z) =  -i \shskip  \bfJ_{\mu + \frac{1}{4}, \shskip m - 1} (z)   +   i \shskip \bfJ_{\mu - \frac{1}{4}, \shskip m + 1} (z)
%\end{cases}
%\end{align}
It follows from \eqref{4eq: J(z)=O(1/|z|)} and \eqref{4eq: recurrence partial} that
\begin{align}\label{4eq: partial J}
\partial \bfJ_{\mu, \shskip m} (z) / \partial   z, \ \partial \bfJ_{\mu, \shskip m} (z)  / \partial \bar z = O  (1 / |z|) , \hskip 15 pt |z| \ra \infty.
\end{align}

\subsubsection{Differential equations}

Define
\begin{align}\label{4eq: nabula}
	\nabla=z^2\frac{\partial^2}{\partial z^2}+z\frac{\partial}{\partial z}+z^2 =\lp z\frac{\partial}{\partial z} \rp^2+z^2.
\end{align} 
and its conjugation
\begin{align}
	\overline{\nabla}=\bar{z}^2\frac{\partial^2}{\partial \bar{z}^2}+\bar{z}\frac{\partial}{\partial \bar{z}}+\bar{z}^2 =\lp \bar{z}\frac{\partial}{\partial \bar{z}} \rp^2+\bar{z}^2.
\end{align}
It is clear from \eqref{appendix: Bessel DE}-\eqref{4eq: defn of J mu m} that $\bfJ_{\mu,\shskip m}(z)$ is an eigenfunction  of both $\nabla$ and $\overline{\nabla}$:
\begin{align}
\label{4eq: nabla J}	\nabla \bfJ_{\mu, \shskip m}(z)& = (  \mu +  {m} / {2} )^2  \bfJ_{\mu, \shskip m}(z),\\
	\overline{\nabla} \bfJ_{\mu, \shskip m}(z)& = (  \mu -  {m}/{2})^2 \bfJ_{\mu, \shskip m}(z).
\end{align}

\subsubsection{An integral representation}

In the polar coordinates, we have the following integral representation of $\bfJ_{\mu, \shskip m}(x e^{i \phi})$ (see \cite[Corollary 6.17]{Qi-Bessel} or \cite[Theorem 12.1]{B-Mo}),
\begin{align}\label{appendix: Integral representation}
	\bfJ_{\mu, \shskip m}(x e^{ i \phi }) = \frac{ 2 (-i)^m}{ \pi} \int_0^{\infty} y^{ 2 \mu - 1}  E ( y e^{ i \phi})^{ - m} J_{m} \left( x Y (y e^{ i \phi}) \right) dy,
\end{align}
with
\begin{align}\label{appendix: Y and E}
	& Y (z) = \left| z + z^{-1}  \right|, %= \lp \lp y + y^{-1} \rp \cos \lp \tfrac 1 2 \phi \rp + \lp y - y^{-1} \rp \sin \lp \tfrac 1 2 \phi \rp \rp^{\frac 1 2}, \\
	\hskip 10 pt E (z) = \lp z + z^{-1} \rp /\left| z + z^{-1}  \right|.
\end{align} The integral in \eqref{appendix: Integral representation} is absolutely convergent only when $|\Re \mu| < \text{\large$\frac 1 4$}$.

%We shall also need the following bound for Bessel functions of purely imaginary order. %For all real $r$ and positive $x$,
%\begin{align}
% J_{\pm 2 i t} (x)  \Lt e^{ \pi t} / \sqrt {x}, \hskip 15 pt x, \, t > 0.
%\end{align}
%with absolute implied constant. 
%In fact, \red{reference}

\subsection{Uniform bounds for  the Bessel function $\bfJ_{it, \shskip m} (z)$}

The following uniform bound for $|z| 	\bfJ_{it, \shskip m} (z)$ will be needed in \S \ref{sec: estimates for Gf} to conclude the proof of Theorem \ref{thm: main}. 
 
\begin{lem}\label{lem: bound 1/z}  We have uniformly
	\begin{align}\label{4eq: bound 1/z}
	|z| \bfJ_{it, \shskip m} (z) \Lt   t^2 + m^2 + 1 .
	\end{align}
\end{lem}

\begin{proof}
It follows from \cite[Proposition 8]{Harcos-Michel} that 
\begin{align}\label{4eq: bound for Jm (x)}
 \sqrt {x} J_{m} (x) \Lt |m| + 1, \quad x > 0.
\end{align} 
By the integral representation (\ref{appendix: Integral representation}, \ref{appendix: Y and E}), we have
\begin{align*}
\left| \bfJ_{it, \shskip m} (x e^{i \phi}) \right| \leqslant \frac{2}{ \pi} \int_0^{\infty}  \left| J_m \big( x Y (y e^{ i \phi}) \big) \right| \frac {dy} y,
\end{align*}
with $Y (y e^{ i \phi})$ defined as in \eqref{appendix: Y and E}. 	Thus \eqref{4eq: bound for Jm (x)} yields 
\begin{align*}
 \bfJ_{it, \shskip m} (x e^{i \phi}) \Lt \frac {|m|+1} {\sqrt{x}} \int_0^{\infty} \frac { d y} {y \sqrt {Y (y e^{i \phi})}  }. 
\end{align*}
Note that we always have
\begin{align*}
Y (y e^{i \phi})   \geqslant | y - 1/ y | ,
\end{align*}
and that the integral
\begin{align*}
\int_0^{\infty} \frac { d y} {y |y-1/y|^{1/2}} 
\end{align*}
is bounded. Hence 
\begin{align}\label{4eq: bound for J, 1}
\bfJ_{it, \shskip m} (x e^{i \phi}) \Lt \frac {|m|+1} {\sqrt{x}} . 
\end{align}

When $|z| \leqslant  t^2 + m^2 +1 $, it follows from \eqref{4eq: bound for J, 1} that
\begin{align*}\label{4eq: bound transition z<}
\bfJ_{it, \shskip m} (z)  & \Lt  \frac {|m|+1} {  \sqrt{|z|}  }  \leqslant  \frac {    t^2 + m^2  + 1  } {  |z|  }.
\end{align*}  
When $|z| >   t^2 + m^2 + 1$, the bound $\bfJ_{it, \shskip m} (z)  \Lt 1 / |z|$ in \eqref{4eq: J(z) <1/z} is more than sufficient. 
\end{proof}

Note that if we use \eqref{4eq: uniform bound, main} with $r = 0$ in Lemma \ref{lem: uniform bound, main} below, the uniform bound \eqref{4eq: bound 1/z} may be improved into 
\begin{equation}\label{4eq: bound 1/z, improved}
|z| \bfJ_{it, \shskip m} (z) \Lt   (t^2 + m^2 + 1)^{2/3}  .
\end{equation}

\vskip 5 pt

The following two lemmas will only be used in Appendix \ref{appendix: Proof of B-M}. So the reader is advised to directly jump to \S \ref{sec: estimates for Gf}.
In these lemmas, the estimates are focused on the  aspect of $t$ and $m$ in order to optimize the condition on $ 	\breve{G} (it , m) $ in Theorem \ref{theorem::Bessel transform--Inversion} or Corollary \ref{cor::Bessel transform--Inversion}. In view of Theorem \ref{thm: analytic}, the reader may find that the optimized condition is far from necessary. 

\begin{lem}\label{lem: bound in t m}
	We have
	\begin{equation}\label{4eq: bound in t m}
	\bfJ_{it, \shskip m} (z) \Lt \left\{\begin{aligned}
& \ds \frac{ \exp \big(  {|z|^2} / {4} \big) /|t| + 1    }{  \sqrt{  t^2 + m^2 }} ,\ & &\text{ if $m$ is even,} \\
& \ds \frac{ \exp \big(  {|z|^2} / {2} \big)}{   \sqrt{  t^2 + m^2 }},\ & & \text{ if $m$ is odd.}
	\end{aligned}\right.  
	\end{equation}
\end{lem}

\begin{proof}
	Let us assume that $(it, m)$ is generic, as otherwise  \eqref{4eq: bound in t m} is trivial.
	
	Recall from \cite[3.13 (1)]{Watson} that for $\nu \neq - 1, -2, -3, ...$
	\begin{align}\label{4eq: Jv = zv...}
	J_{\nu} (z) = \frac {\lp  z /2 \rp^{\nu}} {\Gamma (\nu + 1)} (1 + \theta), 
	\end{align}
	where
	\begin{align}\label{4eq: |theta|, 1}
	|\theta | < \exp \bigg( \frac {|z|^2} {4|\nu_{\oldstylenums{0}} + 1|} \bigg) -1,
	\end{align}
	and  $|\nu_{\oldstylenums{0}} + 1|$ is the smallest  of the numbers $|\nu+1|$, $|\nu+2|$, $|\nu+3|$.... Note that the bound in \eqref{4eq: |theta|, 1} is awful when $ |\nu_{\oldstylenums{0}} + 1| $ is very close to $0$. By modifying the arguments in \cite[\S 2.11]{Watson}, we may also prove
	\begin{align}\label{4eq: |theta|, 2}
	|\theta | < \frac {\exp \lp   {|z|^2} / 4 \rp - 1} {|\nu_{\oldstylenums{0}} + 1|} .
	\end{align}
	We now apply these to $	J_{it, \shskip m} (z)$ as defined in \eqref{1eq: defn of J mu m (z)}. For $\nu = -   it \pm \frac 1 2 m$, we have
	\begin{equation*}%\label{1eq: Bessel function}
	|\nu_{\oldstylenums{0}} + 1| \geqslant \left\{ \begin{aligned}
&  |t|, & & \text{ if $m$ is even,} \\ 
&\left|   it + \tfrac 1 2 \right|, \hskip 5 pt & & \text{ if $m$ is odd.}
	\end{aligned} \right. 
	\end{equation*}   
	Also note that 
	\begin{equation*}%\label{1eq: Bessel function}
	\frac {\left| {\lp   z /2 \rp^{-  it + \frac 1 2 m }}   {\lp  \bar z /2 \rp^{-  it - \frac 1 2 m}} \right| } {\left| \Gamma  \lp -  it - \tfrac{1}{2} m + 1 \rp \Gamma  \lp -  it + \tfrac{1}{2} m + 1 \rp \right|} = \left\{\begin{aligned}
& \ds \frac{ | \sinh (  \pi  t  )|}{ \pi \sqrt{  t^2 + m^2/4}},\ & &\text{ if $m$ is even,} \\
& \ds \frac{  \cosh (  \pi  t  )}{ \pi \sqrt{  t^2 + m^2/4}},\ & &\text{ if $m$ is odd.}
	\end{aligned}\right. 
	\end{equation*}  
	It then follows from \eqref{4eq: Jv = zv...}-\eqref{4eq: |theta|, 2} that 
	\begin{equation*}
\bfJ_{it, \shskip m} (z) \Lt
\left\{ \begin{aligned}
& \ds \frac{ \min \left\{ \exp \big(  {|z|^2} / {4 |t| } \big), \exp \big(  {|z|^2} / {4} \big) / |t| + 1 \right\} }{ \sqrt{  t^2 + m^2 }} ,\ & &\text{ if $m$ is even,} \\
& \ds \frac{ \exp \big(  {|z|^2} / {|4 it +2|} \big)}{   \sqrt{  t^2 + m^2 }},\ & & \text{ if $m$ is odd.}
\end{aligned}\right.  
	\end{equation*}
This clearly implies \eqref{4eq: bound in t m}.
\end{proof}	

%The following lemma is very useful, especially in the proof of Theorem \ref{thm: analytic}. We shall exploit  in the proof of this lemma some deep results of L. J. Landau for classical Bessel functions. It should be remarked that  weaker versions of this lemma obtained from more well-known and classical results (for example \cite[8.43 (1, 2)]{Watson} and \cite[\S 3.14.2, 3.14.3]{MO-Formulas}) would also be sufficient for our purpose.
	
\begin{lem}\label{lem: uniform bound, main}
	For any given $0 \leqslant r < \text{\large$\frac 1 3$} $, we have
	\begin{align}\label{4eq: uniform bound, main}
	\bfJ_{it, \shskip m} (z) \Lt_{\shskip r} \frac 1 { (|m|+1)^{r} |z|^{ 1/3 - r} } . 
	\end{align}
\end{lem}

\begin{proof}
	
	For this lemma, we need some uniform bounds for the classical Bessel function $J_{\nu} (x)$ of real argument.
	First of all, it is well known that for all $x $ real and $\nu \geqslant 0$,
	\begin{align*}
		& |J_{\nu} (x)| \leqslant 1.
	\end{align*} 
	See for example \cite[9.1.60]{A-S}.
	Moreover, the following bounds are obtained by L. J. Landau \cite{Landau-Bessel} in his study of monotonicity properties of Bessel functions,
	\begin{align*}
		|J_{\nu} (x)|  <  b/ \nu^{ 1/3}, \hskip 10 pt |J_{\nu} (x)|  \leqslant  c/|x|^{1/3},
	\end{align*}
	with $b = 0.674885...$ and $c = 0.785746...$. Combining the bounds above, we deduce easily that, for any  $0 \leqslant r \leqslant \text{\large$\frac 1 3$}$,
	\begin{align}\label{4eq: bound for Jv(x)}
		J_{\nu} (x)  \Lt_{\shskip r} \frac 1 { (  \nu + 1)^{ r } |x|^{{ 1/3 -r}}}, \hskip 15 pt x \text{ real, } \nu \geqslant 0.
	\end{align}
	
We now return to the proof. By the arguments that lead us to \eqref{4eq: bound for J, 1}, if we apply \eqref{4eq: bound for Jv(x)} in place of \eqref{4eq: bound for Jm (x)} then 
\begin{align*}
\bfJ_{it, \shskip m} (x e^{i \phi}) & \Lt_{\shskip r} \frac 1 { ( |m| + 1)^{r} x^{1/3 - r}} \int_0^{\infty} \frac { d y} {y Y (y e^{i \phi})^{1/3 - r}} \\
& \leqslant \frac 1 { ( |m| + 1)^{r} x^{1/3 - r}} \int_0^{\infty} \frac { d y} {y |y-1/y|^{1/3 - r}} \\
& \Lt_{r} \frac 1 { ( |m| + 1)^{r} x^{1/3 - r}},
\end{align*}
provided that $ 0 \leqslant r < \text{\large$\frac 1 3$}$. %Then follows \eqref{4eq: uniform bound, main} and the lemma.
\end{proof}

\section{Completion of the proof of Theorem \ref{thm: main}}\label{sec: estimates for Gf}

To conclude the proof of Theorem \ref{thm: main}, we need to prove that $G_f (z)$ satisfies the requirements
of Corollary \ref{cor::Bessel transform--Inversion} so that Bessel inversion holds and that the change of order of the limit, summation
and integration which we employed in the proof is justified.  

\subsection{Bounds on the Bessel transform} \label{sec: bounds for Gf}

First, we prove the following estimate for the Bessel transform $\breve{G}_{f} (it ,  m)$ as defined in \eqref{2eq:Bessel transform--definition}. Our proof can be easily adapted in the real case so that the estimates in \cite{BaruchMao-Whittaker} may be considerably improved. This is based on the idea of applying the differential operator $\nabla$ (instead of recurrence relations as in  \cite{BaruchMao-Whittaker}) and an observation of the anonymous referee as in Lemma \ref{lem: referee's observation}. 

\begin{thm}\label{thm: analytic}
	We have $ \breve{G}_{f} (it ,  m) =  O \big( 1 / ( t^2 + m^2 + 1 )^{A} \big)  $ 
	for   arbitrary $A \geqslant 0$, with the implied constant depending on $A$ and $f$.
\end{thm}

First of all, we need a simple formula for partial integration. It will be very convenient to integrate on differential forms in terms of $z$ and $\bar z$. In particular, the volume form  on $\BC$ is $i \shskip d z   \nwedge  d \bar z  $ (there is a slight abuse of notation as $d z$ was the Haar measure on $\BC$). By Stokes' theorem,  one can easily derive the following formula; the proof is left to the reader as an exercise.
\begin{lem}\label{lem: Integration by parts}
	Let $F$, $G \in C^{\infty}(\BCO)$. Assume that $
	F(z) G(z)  =  o(1)$ for both $|z| \ra 0$ and $|z| \ra \infty$. 
	We have
	\begin{equation*}
	\sideset{}{_{\BC \smallsetminus\{0\}}}{\iint}   z \frac{\partial F (z)}{\partial z} \cdot G(z) \shskip \shskip \frac {i \shskip d z \nwedge d \bar z} {z \bar z} = - \sideset{}{_{\BC \smallsetminus\{0\}}}{\iint}  F(z) \cdot z \frac{\partial G (z)}{\partial z} \shskip \frac {i \shskip d z \nwedge d \bar z} {z \bar z},
	\end{equation*}
	provided that one of  the integrals is convergent.
\end{lem}
%In view of the bounds for $\bfJ_{it, \shskip m} (z)$ and $\partial \bfJ_{it, \shskip m} (z) /\partial z$ in  \eqref{4eq: J(z) <1/z} and \eqref{4eq: partial J}, we have the following corollary. 

\begin{cor}\label{cor: Integration by parts}
    Let $\nabla$ be defined as in {\rm\eqref{4eq: nabula}}. Let $G (z)$ be a smooth function that vanishes around $0$ and such that 
    $G (z) = o (1)$, $ \partial  G (z) / \partial z = o (1)$ for $|z| \ra \infty$.
	We have
	\begin{align}\label{5eq: Integration by parts}
	\int_{\BC \smallsetminus\{0\}}   \bfJ_{it, \shskip m}(z) \cdot  G(z) \dx z = \frac 1 {(   i t + m / 2)^2}  \int_{\BC \smallsetminus\{0\}}  \bfJ_{it, \shskip m}(z) \cdot \nabla G(z) \dx z ;
	\end{align}
 the integral on the left is convergent. 
\end{cor}

\begin{proof} The identity \eqref{5eq: Integration by parts} follows from two applications of Lemma \ref{lem: Integration by parts}, along with  
	\begin{align*}
	\nabla \bfJ_{it, \shskip m} (z) = (  i t + m / 2)^2 \bfJ_{it, \shskip m}(z) .
	\end{align*}
See \eqref{4eq: nabla J}. Note that the condition of Lemma  \ref{lem: Integration by parts} may be verified by  the bounds for $\bfJ_{it, \shskip m} (z)$ and $\partial \bfJ_{it, \shskip m} (z) /\partial z$  in  \eqref{4eq: J(z)=O(1/|z|)} and \eqref{4eq: partial J} together with the conditions on $G(z)$.
\end{proof}

The following lemma is due to the anonymous referee. 

\begin{lem}\label{lem: referee's observation}
	Define $\mathscr{V}$ to be the space of functions on $\BC \smallsetminus \{0\}$ of the form
	\begin{align}
	G (z) = \frac {e ((z + \widebar{z} )/2\pi)  } {|z|} H_{+} (1/z) + \frac {e (- (z + \widebar{z} )/2\pi)  } {|z|} H_{-} (1/z),
	\end{align}
	where $H_{\pm} \in C_c^{\infty} (\BC)$. Then $ \mathscr{V} $ is stable by the differential operator $\nabla$. Moreover, if $G (z)$ is a function in $\mathscr{V}$ as above then
	$G (z) = O (1/|z|)$, $ \partial  G (z) / \partial z = O  (1/|z| )$ for $|z| \ra \infty$.
\end{lem}

\begin{proof}
	By direct calculations, we find that
	\begin{align*}
	\frac{\partial}{\partial z} \left\{  \frac {e (\pm (z + \widebar{z} )/2\pi)  } {|z|} H_{\pm} (1/z) \right\} & = \frac {e (\pm (z + \widebar{z} )/2\pi)  } {|z|} \\
	\cdot & \lp    \lp - \frac 1 { 2 z } \pm   i   \rp  H_{\pm} (1/z) - \frac 1 {z^2} H_{\pm}^1 (1/z) \rp,
	\end{align*}
	and
	\begin{align*}
	\nabla \left\{  \frac {e (\pm (z + \widebar{z} )/2\pi)  } {|z|} H_{\pm} (1/z) \right\} &  = \frac {e (\pm (z + \widebar{z} )/2\pi)  } {|z|} \\
	\cdot & \lp \frac 1 4 H_{\pm} (1/z) +   2 \lp \frac 1 z \mp i \rp H_{\pm}^1 (1/z) + \frac 1 {z^2} H_{\pm}^2 (1/z) \rp,
	\end{align*}
	with 
	\begin{align*}
	H^1_{\pm} (z) = \partial H_{\pm} (z)/\partial z  , \quad H^2_{\pm} (z) =  \partial^2 H_{\pm} (z) /\partial z^2 . 
	\end{align*}
	Then our assertions become obvious. 
\end{proof}

%In view of \eqref{2eq:definition--Gf} and \eqref{3eq: Orbit integral--characterization},  the function $G_f (z)$ is %in the space $\mathscr{V}$. Theorem \ref{thm: analytic} follows immediately from applying Corollary \ref{cor: %Integration by parts} repeatedly. 

%Now we prove Theorem \ref{thm: analytic}. 

In view of \eqref{2eq:definition--Gf} and Theorem \ref{thm: Jacquet}, we have $G_f \in \mathscr{V}$. Hence Lemma \ref{lem: referee's observation} implies that $\nabla^{k}G_f \in \mathscr{V}$ for any integer $k \geqslant 0$ and that Corollary \ref{cor: Integration by parts} is applicable for each $\nabla^{k}G_f$. Theorem \ref{thm: analytic} follows immediately from applying Corollary \ref{cor: Integration by parts} repeatedly, along with a final estimation by Lemma \ref{lem: bound 1/z}.

From Theorem \ref{thm: analytic} we infer that the conditions in Corollary \ref{cor::Bessel transform--Inversion} are satisfied by $G_f (z)$.

\subsection{Change of order of limit, summation and integration} \label{sec: change of order}

Next, we prove that the expression inside the limit in \eqref{3eq: lim sum and int} is  convergent absolutely and uniformly in $z$. %We unify the even case and the odd case in an obvious way. 
By \eqref{4eq: bound 1/z}, we have
\begin{align*}%\label{5eq: bound  zJ}
\frac { |z| \bfJ_{it, \shskip 2 d} (z) } { \cos (z+\bar{z}) }
\Lt   t^2 + d^2 + 1   , 
\end{align*}
which is uniform on the set of $z$ such that $|\cos (z+\bar{z})| > \text{\Large$\tfrac 1 2$}$ say. By Theorem \ref{thm: analytic} with $A = 4$ say, we have
\begin{align*}
\breve{G}_{f} (it ,  2 d) \Lt \frac { 1} { ( t^2 + d^2 + 1 )^4    }.
\end{align*}
Consequently,
\begin{align*}
   \sum_{d=-\infty}^{\infty}\hskip -1pt \int_{-\infty}^{\infty} \hskip -1pt \frac{   \left|z \shskip \bfJ_{it, \shskip 2d} (z) \right|}{ |  \cos (z+\bar{z})| }  \left|\breve{G}_f (it, 2d) \right| \big( t^2  + d^2 \big) dt 
  \Lt  
\sum_{d = -\infty} ^{\infty} \hskip -1pt  \int_{  -\infty}^{\infty} \hskip -1pt \frac{ dt }{ ( t^2 + d^2 + 1)^{ 2 } } < \infty.
\end{align*}
Hence we can use the dominated convergence theorem to move the limit into the sum and integral in \eqref{3eq: lim sum and int}. Similar is the odd case.

\appendix

\section{Proof of Theorem \ref{theorem::Bessel transform--Inversion}}\label{appendix: Proof of B-M}
    
    Recall the Bessel transform defined in \eqref{2eq:Bessel transform--definition}. We have the following Parseval-Plancherel formula due to Bruggeman, Motohashi and Lokvenec-Guleska (\cite[(11.3)]{B-Mo} and \cite[Proposition 12.2.2]{B-Mo2}),
    \begin{align}\label{appendix: Parseval}
    \int_{\BC \smallsetminus\{0\}} G (z) \overline{F(z)} \dx z = \frac{1}{8} \sum_{m = -\infty}^{\infty} \int_{-\infty}^{\infty}   \breve{G} (it , m) \overline {\breve{F} (it , m)}  (t^2 + m^2/4   )d t,
    \end{align}
    for functions $G$, $F \in C_c^{\infty} (\BCO)$. Observe that
    \begin{align*}
    \overline{\breve{F} (i t, m)} = (-)^m \shskip \breve{\xoverline[0.75]{F\hskip 1 pt}} \hskip -1 pt (i t , m) .
    \end{align*}
%    where bars denote complex conjugates. 
It follows from \eqref{appendix: Parseval} that the Bessel transform establishes an (abstract) isometry between the Hilbert spaces 
\begin{align*}
L^2 (\BCO,   \dx z) \longrightarrow L^2 (\BZ \times \BR, ( t^2 + m^2 / 4 )dt / 8).
\end{align*}
Now let $G (z) \in  L^1 (\BCO, \dx z ) \cap L^2 (\BCO, \dx z)$ and $F (z) \in C_c^{\infty} (\BCO)$ so that their Bessel transform $\breve{G} (it , m)$ and $\breve{F} (it , m)$ given as in \eqref{2eq:Bessel transform--definition} are both well defined. We may then write \eqref{appendix: Parseval} in the following way
    \begin{align*}%\label{appendix: Parseval noncompact}
    \int_{\BC \smallsetminus\{0\}} & G(z) \overline{F(z)} \dx z  = \frac{1}{8} \sum_{m = -\infty}^{\infty} \int_{-\infty}^{\infty}  \int_{\BC\smallsetminus\{0\}} (-)^m \breve{G} (it , m)   \bfJ_{it, \shskip m}(z)  (  t^2 + m^2/4 )  \overline {F(z)} \dx z \shskip d t.
    \end{align*}
Thus the Bessel inversion formula \eqref{2eq::Bessel transform--Inversion formula} in Theorem \ref{theorem::Bessel transform--Inversion} follows immediately if $G (z)$ is continuous and the right hand side converges absolutely. This in turn follows from our assumption $	\breve{G} (it , m) = O\bigl( 1/ (t^2+ m^2 + 1)^{q} (|m|+1) \bigr)$ ($q > 1$) in  Theorem \ref{theorem::Bessel transform--Inversion}, along with the uniform bounds for $ 	\bfJ_{it, \shskip m} (z) $ in Lemma \ref{lem: bound in t m} and \ref{lem: uniform bound, main}. In fact, %to show the convergence we truncate the integral on the right hand side of \eqref{2eq::Bessel transform--Inversion formula} at $|t| = 1$. 
by the estimates in \eqref{4eq: bound in t m} and \eqref{4eq: uniform bound, main}, we have
\begin{equation*}
\bfJ_{it, \shskip m} (z)  \Lt \left\{ \begin{split}
&1 / (|t| + |m|), \hskip 5 pt && \text{ if $|t| > 1$,} \\
& 1, && \text{ if $|t| \leqslant 1$,}
\end{split} \right.   
\end{equation*}
which is uniform on any given compact subset of $\BCO$, say on the support of $F $. Hence the sum of integrals on the right hand side of \eqref{2eq::Bessel transform--Inversion formula} is uniformly and absolutely bounded by
\begin{align*}
\sum_{m= - \infty}^{\infty}  \frac{1}{|m|+1 } \lp \int_{ -\infty}^{-1} + \int_{1}^{\infty} \frac { dt}{(|t| + |m|)^{2q-1}} +  \int_{-1}^{1}  \frac {  dt}{(|m| + 1 )^{2q-2}} \rp < \infty,
\end{align*}
as desired.

%Our assumption on $\breve{G}$ then guarantees that the right side of \eqref{2eq::Bessel transform--Inversion formula} converges absolutely and locally uniformly, and hence defines a continuous function on $\BCx$. Denote this function by $\tilde{G}$. Then, a simple change of order of integration shows that 

\section{Representation theory of Bessel functions and distributions for $\SL_2 (\BC)$}\label{appendix: Bessel}
    
    In this appendix, we give a brief review of Bessel distributions and Bessel functions for $\SL_2 (\BC)$ in \cite{Chai-Qi-Bessel}. For additional results on Bessel distributions and Bessel functions see \cite{Baruch-GL2}-\cite{B-Mo-Kernel2}, \cite{CPS} and \cite{Mo-Kernel2}.  \nocite{Baruch-Split,BaruchMao-NA,BaruchMao-Real}

%\subsection{Representations of  $\SL_2 (\BC)$} 

\subsection{Bessel distributions  for $\SL_2 (\BC)$}\label{sec: Bessel distributions}

The Bessel distribution $J_{\pi, \shskip \psi}$ given by \eqref{1eq: defn of J mu m (z)}-\eqref{1eq: Bessel coefficient} were defined in an {\it ad hoc} way. The  definition of  $J_{\pi, \shskip \psi}$ in representation theory is another story. 

Let $G = \SL_2 (\BC)$ and $K = \SU_2 (\BC)$. Let $\pi$ be an infinite dimensional irreducible unitary
representation of $G$  on a Hilbert space $H$ and $\langle \,  ,\shskip \rangle $ be a $G$-invariant nonzero inner product on $H$. Let $H_\infty$ be the subspace of smooth vectors in $H$. Let $\psi = \psi_{\lambdaup}$ be an additive character as defined in \eqref{1eq: defn of psi}. 
It is well known that there exists a  nonzero continuous $\psi$-Whittaker
functional $L$ on $H_\infty$, unique up to scalars, satisfying
\begin{equation*}
L (\pi(n(u) )\varv)=\psi( u )L(\varv), \hskip 15 pt n (u) \in N, \   \varv\in H_\infty.
\end{equation*} 
Let
\begin{equation*}%\label{def: Wv(g)}
W_\varv(g)=L(\pi(g)\varv), \hskip 15 pt  \varv\in H_\infty, \  g\in G,
\end{equation*}
be the Whittaker function corresponding to $\varv$. 
We normalize the Whittaker functional $L$ so that
\begin{equation*}%\label{eq: inner product}
\langle \varv_1,\varv_2\rangle =\int_{\BCx}W_{\varv_1}(s(z))\overline
{W_{\varv_2}(s(z))}\dx z.
\end{equation*}
It is well known that for every $\varv_1,\varv_2 \in H_{\infty}$ the integral above is absolutely convergent and indeed gives a $G$-invariant inner product on $H_{\infty}$.

For every $f \in C^{\infty}_c (G)$ and $\varv \in H$ we define 
\begin{align*}
\pi (f) \varv = \int_G f(g) \pi (g) \varv \, d g,
\end{align*} 
where $d g$ is the Haar measure on $G$ defined after \eqref{1eq: Bessel coefficient}. 

Let  $\{ \varv_{\mathfrak i} \}$ be an orthonormal basis in $H_{\infty}$ such that each $\varv_{\mathfrak i}$ is contained in a $K$-isotypic component (see \cite[Lemma 8.1.1]{RRG-I} for this condition). We define the Bessel distribution $ J_{\pi,\shskip \psi} $ by 
\begin{align}\label{3eq: J(f)}
	J_{\pi,\shskip \psi} (f) = \sum_{\mathfrak i}  L (\pi(f)\varv_{\mathfrak i}) \overline {L (\varv_{\mathfrak i})}.
\end{align}
The distribution is independent on the choice of such orthonormal basis. 

The main result on Bessel distributions for $\SL_2 (\BC)$ is a regularity theorem in \cite{Chai-Qi-Bessel}. 

\begin{thm}\label{thm::Bessel distribution}
	Let $\pi$ be an irreducible unitary representation of $G$. Then there exists a real analytic function $j_{\pi,\shskip \psi} : B \varw B \ra \BC$ which is locally
	integrable on G such that
	\begin{align*}
	J_{\pi,\psi}(f)=\int_G f(g)j_{\pi,\psi}(g)dg,  
	\end{align*}
	for all $  f\in C_c^{\infty}(G)$.
\end{thm}

\subsection{Bessel functions for $\SL_2 (\BC)$} 

The interpretation in representation theory of Bessel functions $j_{\pi, \shskip \psi}$ defined as in \eqref{1eq: defn of J mu m (z)}-\eqref{1eq: Bessel function--NN rule} is the following kernel formula.

\begin{thm}
	Let $\pi$ be an irreducible unitary representation of $G$. We have
\begin{align*}
W_\varv(s(y)\varw)=\int_{\BCx}j_{\pi, \shskip \psi} (s(y z)\varw) W_\varv(s(z)) \dx z
\end{align*}
for all $\varv\in H_{\infty}$.
\end{thm}

This kernel formula is proven in \cite[Chapter 4]{Qi-Bessel} in full generality; some of its special cases may be found in \cite{B-Mo-Kernel2,Mo-Kernel2} and \cite{Baruch-Kernel}. It is first introduced in \cite{CPS} for $\SL_2 (\BR)$. 

It is proven in \cite{Chai-Qi-Bessel} that the function $j_{\pi,\shskip \psi}$ in Theorem \ref{thm::Bessel distribution} is exactly the Bessel function defined as in \eqref{1eq: defn of J mu m (z)}-\eqref{1eq: Bessel function--NN rule}. Thus coincide the two Bessel distributions $J_{\pi,\shskip \psi}$ defined in the \nameref{sec: Intro}  and \S \ref{sec: Bessel distributions}. 

%It does depend on the choice of the Whittaker functional $L$, the $G$ invariant inner product $\langle \,  ,\shskip \rangle $ and the Haar measure $d g$. Hence $ J_{\pi,\shskip \psi} $ is defined up to a positive constant. 

%	\bibliographystyle{alphanum}
%	\bibliography{references}

\begin{thebibliography}{vdBK}
	
	\bibitem[AS]{A-S}
	M.~Abramowitz and I.~A. Stegun.
	\newblock {\em Handbook of {M}athematical {F}unctions with {F}ormulas,
		{G}raphs, and {M}athematical {T}ables},   {\em National Bureau of
		Standards Applied Mathematics Series}, 55.
	\newblock Washington, D.C., 1964.
	
	\bibitem[Bar1]{Baruch-GL2}
	E.~M. Baruch.
	\newblock On {B}essel distributions for {${\rm GL}_2$} over a {$p$}-adic field.
	\newblock {\em J. Number Theory}, 67(2):190--202, 1997.
	
	\bibitem[Bar2]{Baruch-Split}
	E.~M. Baruch.
	\newblock On {B}essel distributions for quasi-split groups.
	\newblock {\em Trans. Amer. Math. Soc.}, 353(7):2601--2614, 2001.
	
	\bibitem[BBA]{Baruch-Kernel}
	E.~M. Baruch and O.~Beit-Aharon.
	\newblock A kernel formula for the action of the {W}eyl element in the
	{K}irillov model of {$\mathrm{SL}(2,\mathbb{C})$}.
	\newblock {\em J. Number Theory}, 146:23--40, 2015.
	
	\bibitem[BM1]{BaruchMao-NA}
	E.~M. Baruch and Z.~Mao.
	\newblock Bessel identities in the {W}aldspurger correspondence over a
	{$p$}-adic field.
	\newblock {\em Amer. J. Math.}, 125(2):225--288, 2003.
	
	\bibitem[BM2]{BaruchMao-Real}
	E.~M. Baruch and Z.~Mao.
	\newblock Bessel identities in the {W}aldspurger correspondence over the real
	numbers.
	\newblock {\em Israel J. Math.}, 145:1--81, 2005.
	
	\bibitem[BM3]{BaruchMao-Whittaker}
	E.~M. Baruch and Z.~Mao.
	\newblock A {W}hittaker-{P}lancherel inversion formula for {${\rm SL}(2,{\mathbb{R}})$}.
	\newblock {\em J. Funct. Anal.}, 238(1):221--244, 2006.
	
	\bibitem[BM4]{B-Mo-Kernel2}
	R.~W. Bruggeman and Y.~Motohashi.
	\newblock A note on the mean value of the zeta and {$L$}-functions. {XIII}.
	\newblock {\em Proc. Japan Acad. Ser. A Math. Sci.}, 78(6):87--91, 2002.
	
	\bibitem[BM5]{B-Mo}
	R.~W. Bruggeman and Y.~Motohashi.
	\newblock Sum formula for {K}loosterman sums and fourth moment of the
	{D}edekind zeta-function over the {G}aussian number field.
	\newblock {\em Funct. Approx. Comment. Math.}, 31:23--92, 2003.
	
	\bibitem[BP1]{B-P-1}
	R.~Beuzart-Plessis.
	\newblock A local trace formula for the {G}an-{G}ross-{P}rasad conjecture for
	unitary groups: the {A}rchimedean case.
	\newblock {\em arXiv:1506.01452, to appear in Ast\'{e}risque}, 2015.
	
	\bibitem[BP2]{B-P-2}
	R.~Beuzart-Plessis.
	\newblock Plancherel formula for {$\mathrm{GL}_n (F) \backslash \mathrm{GL}_n
		(E) $} and applications to the {I}chino-{I}keda and formal degree conjectures
	for unitary groups.
	\newblock {\em   arXiv:1812.00047, preprint}, 2018.
	
	
	
	\bibitem[CPS]{CPS}
	J.~W. Cogdell and I.~Piatetski-Shapiro.
	\newblock {\em The {A}rithmetic and {S}pectral {A}nalysis of {P}oincar\'e
		{S}eries}.
	\newblock Perspectives in Mathematics, Vol. 13. Academic Press, Inc., Boston,
	MA, 1990.
	
	\bibitem[CQ]{Chai-Qi-Bessel}
	J.~Chai and Z.~Qi.
	\newblock Bessel identities in the {W}aldspurger correspondence over the
	complex numbers.
	\newblock {\em arXiv:1802.01229, to appear in Israel J. Math.}, 2018.
	
	\bibitem[Del]{Delorme-W-P}
	P.~Delorme.
	\newblock Formule de {P}lancherel pour les fonctions de {W}hittaker sur un
	groupe r\'{e}ductif {$p$}-adique.
	\newblock {\em Ann. Inst. Fourier (Grenoble)}, 63(1):155--217, 2013.
	
	
	\bibitem[HM]{Harcos-Michel}
	G.~Harcos and P.~Michel.
	\newblock The subconvexity problem for {R}ankin-{S}elberg {$L$}-functions and
	equidistribution of {H}eegner points. {II}.
	\newblock {\em Invent. Math.}, 163(3):581--655, 2006.
	
	
	\bibitem[Jac]{Jacquet-RTF}
	H.~Jacquet.
	\newblock On the nonvanishing of some {$L$}-functions.
	\newblock {\em Proc. Indian Acad. Sci. Math. Sci.}, 97(1-3):117--155 (1988),
	1987.
	
	\bibitem[Kna]{Knapp-Book}
	A.~W. Knapp.
	\newblock {\em Representation {T}heory of {S}emisimple {G}roups, an {O}verview
		{B}ased on {E}xamples},   {\em Princeton Mathematical Series}, 36.
	\newblock Princeton University Press, Princeton, NJ, 1986.
	
	\bibitem[Kuz]{Kuznetsov}
	N.~V. Kuznetsov.
	\newblock {P}etersson's conjecture for cusp forms of weight zero and {L}innik's
	conjecture. {S}ums of {K}loosterman sums.
	\newblock {\em Math. Sbornik}, 39:299--342, 1981.
	
	\bibitem[Lan]{Landau-Bessel}
	L.~J. Landau.
	\newblock Bessel functions: monotonicity and bounds.
	\newblock {\em J. London Math. Soc. (2)}, 61(1):197--215, 2000.
	
	\bibitem[LG]{B-Mo2}
	H.~Lokvenec-Guleska.
	\newblock {\em {S}um {F}ormula for $\mathrm{SL}_2$ over {I}maginary {Q}uadratic
		{N}umber {F}ields}.
	\newblock Ph.D. Thesis. Utrecht University, 2004.
	
	\bibitem[Mot]{Mo-Kernel2}
	Y.~Motohashi.
	\newblock Mean values of zeta-functions via representation theory.
	\newblock In {\em Multiple {D}irichlet {S}eries, {A}utomorphic {F}orms, and
		{A}nalytic {N}umber {T}heory},   {\em Proc. Sympos. Pure Math.}, 75, 
	pages 257--279. Amer. Math. Soc., Providence, RI, 2006.
	
	
	
	\bibitem[Olv]{Olver}
	F.~W.~J. Olver.
	\newblock {\em Asymptotics and {S}pecial {F}unctions}.
	\newblock Academic Press, New York-London, 1974. 
	
	\bibitem[Qi1]{Qi-Thesis}
	Z.~Qi.
	\newblock {\em {T}heory of {B}essel {F}unctions of {H}igh {R}ank}.
	\newblock Ph.D. Thesis. The Ohio State University, 2015.
	
	\bibitem[Qi2]{Qi-Bessel}
	Z.~Qi.
	\newblock Theory of fundamental {B}essel functions of high rank.
		\newblock {\em arXiv:1612.03553, to appear in Mem. Amer. Math. Soc.}, 2016.
	
	\bibitem[Qi3]{Qi-Kuz}
	Z.~Qi.
	\newblock On the {K}uznetsov trace formula for {$\mathrm{PGL}_2(\Bbb{C})$}.
	\newblock {\em J. Funct. Anal.}, 272(8):3259--3280, 2017.
	
	
	\bibitem[Qi4]{Qi-II-G}
	Z.~Qi.
	\newblock On the {F}ourier transform of {B}essel functions over complex
	numbers---{II}: the general case. 
	\newblock {\em Trans. Amer. Math. Soc.}, 372:2829--2854, 2019.
	
	\bibitem[ST]{Sears-Titchmarsh}
	D.~B. Sears and E.~C. Titchmarsh.
	\newblock Some eigenfunction formulae.
	\newblock {\em Quart. J. Math., Oxford Ser. (2)}, 1:165--175, 1950.
	
	\bibitem[SV]{Sak-Venkatesh-Periods}
	Y.~Sakellaridis and A.~Venkatesh.
	\newblock Periods and harmonic analysis on spherical varieties.
	\newblock {\em Ast\'{e}risque}, (396):viii+360, 2017.
	
	\bibitem[Tit]{Titchmarsh-Eigenfunction}
	E.~C. Titchmarsh.
	\newblock {\em Eigenfunction {E}xpansions {A}ssociated with {S}econd-{O}rder
		{D}ifferential {E}quations. {P}art {I}}.
	\newblock Second Edition. Clarendon Press, Oxford, 1962.
	
	\bibitem[vdBK]{vdBK-Wallach}
	E.~P. van~den Ban and J.~J. Kuit.
	\newblock Cusp forms for reductive symmetric spaces of split rank one.
	\newblock {\em Represent. Theory}, 21:467--533, 2017.
	
	\bibitem[Wal1]{RRG-I}
	N.~R. Wallach.
	\newblock {\em Real {R}eductive {G}roups. {I}}, {\em Pure and
		Applied Mathematics}, 132.
	\newblock Academic Press, Inc., Boston, MA, 1988.
	
	\bibitem[Wal2]{RRG-II}
	N.~R. Wallach.
	\newblock {\em Real {R}eductive {G}roups. {II}}, {\em Pure and
		Applied Mathematics}, 132.
	\newblock Academic Press, Inc., Boston, MA, 1992.
	
	\bibitem[Wal3]{Wallach-Correction}
	N.~R. Wallach.
	\newblock On the {W}hittaker {P}lancherel theorem for real reductive groups.
	\newblock {\em  arXiv:1705.06787, preprint}, 2017.
	
	\bibitem[Wat]{Watson}
	G.~N. Watson.
	\newblock {\em A {T}reatise on the {T}heory of {B}essel {F}unctions}.
	\newblock Cambridge University Press, Cambridge, England; The Macmillan
	Company, New York, 1944.
	
\end{thebibliography}

\end{document}